\documentclass[11pt,twoside, final]{amsart}

\copyrightinfo{0}{Iranian Mathematical Society}
\usepackage{amsmath,amsthm,amscd,amsfonts,amssymb,enumerate}
\usepackage{graphicx}		
\usepackage{color}
\usepackage[colorlinks]{hyperref}
 
\newtheorem{theorem}{Theorem}[section]

\newtheorem{lemma}[theorem]{Lemma}
\newtheorem{corollary}[theorem]{Corollary}
\theoremstyle{definition}

\theoremstyle{algorithm}
\newtheorem{algorithm}[theorem]{Algorithm}
\theoremstyle{remark}
\newtheorem{remark}[theorem]{Remark}
\numberwithin{equation}{section}
 
 \commby{Hamid Pezeshk}
\begin{document}

 
\title[Common solutions to equilibrium problems]{THE COMMON SOLUTIONS TO PSEUDOMONOTONE EQUILIBRIUM PROBLEMS} 
 
\author[D. V. HIEU]{DANG VAN HIEU$^*$}
\address[\textbf{DANG VAN HIEU}]{Department of Mathematics, Ha Noi University of Science, VNU.
334 - Nguyen Trai Street - Ha Noi - Viet Nam}
\email{dv.hieu83@gmail.com}


  \thanks{$^*$Corresponding author}
%
\maketitle
%

\begin{abstract}
In this paper, we propose two iterative methods for finding a common solution of a finite family of equilibrium problems 
for pseudomonotone bifunctions. The first is a parallel hybrid extragradient-cutting algorithm which is extended from the 
previously known one for variational inequalities to equilibrium problems. The second is a new cyclic hybrid 
extragradient-cutting algorithm. In the cyclic algorithm, using the known techniques, we can perform and develop practical numerical experiments.\\
\textbf{Keywords:}  Hybrid method, parallel algorithm, cyclic algorithm, extragradient method, equilibrium problem.  \\
\textbf{MSC(2010):}  Primary: 90C33; Secondary: 68W10, 65K10.
\end{abstract}
\section{\bf Introduction}\label{intro}
Let $C$ be a nonempty closed and convex subset of a real Hilbert space $H$ and $f$ be a bifunction from $H\times H$ to the set of real numbers $\mathbb{R}$. The equilibrium problem (EP) for the bifunction $f$ on $C$ is to find $x^*\in C$ such that
\begin{equation}\label{eq:EP}
f(x^*,y)\ge 0,~\forall y\in C.
\end{equation}
The solution set of the EP $(\ref{eq:EP})$ is denoted by $EP(f)$. The EP is a generalization of many mathematical problems \cite{BO1994,MO1992}. 
In recent years, many algorithms have been proposed for solving the EP, see \cite{A2011,BO1994,DHM2014,M2000,MO1992,QMH2008,TT2007} 
and the references therein. When the bifunction $f$ is monotone, the most of existing algorithms for solving the EP involve the regularization equilibrium 
problem (REP), i.e., at the $n^{th}$ iteration step, known $x_n$, determine the next approximation $x_{n+1}$ as the solution of the problem:
\begin{equation}\label{eq:RugularEP}
\mbox{Find} ~x\in C~ \mbox{such that:}~ f(x,y)+\frac{1}{r_n}\left\langle y-x, x-x_n\right\rangle\ge 0,~\forall y\in C, 
\end{equation}
where $r_n\ge d>0$. Note that the problem $(\ref{eq:RugularEP})$ is strongly monotone when the bifunction $f$ is monotone. 
Thus, its solution exists and is unique under certain assumption of the continuty of the bifunction $f$. Unforturnately, in general, 
for instance $f$ is pseudomonotone, the problem $(\ref{eq:RugularEP})$ is not strongly monotone and so the unique solvability of 
$(\ref{eq:RugularEP})$ is not guaranteed even its solution set can not be convex. In this case, the authors in \cite{A2011,QMH2008} 
replaced the REP $(\ref{eq:RugularEP})$ by two strong convex programs
$$
\left \{
\begin{array}{ll}
y_n=\arg\min\left\{\rho f(x_n,y)+\frac{1}{2}||x_n-y||^2: y\in C\right\},\\
x_{n+1}=\arg\min\left\{\rho f(y_n,y)+\frac{1}{2}||x_n-y||^2: y\in C\right\}, 
\end{array}
\right.
$$
where $\rho>0$ and satisfies some suitable conditions.

Now let $K_i,i=1,\ldots,N$ be a finite family of closed and convex subsets of $H$ such that $K=\cap_{i=1}^N K_i \ne \O$ 
and $f_i:H\times H\to \mathbb{R},i=1,\ldots,N$ be pseudomonotone bifunctions. The problem, so called the common solutions to 
equilibrium problems (CSEP), for the bifunctions $f_i$ is stated as follows: Find $x^*\in K$ such that
\begin{equation}\label{eq:CSEP}
f_i(x^*,y)\ge 0, ~\forall y\in K_i, ~i=1,\ldots,N.
\end{equation}
Clearly, the CSEP with $N=1$ is the EP. The motivation and inspiration for researching the CSEP with $N>1$ are originated from 
some simple observations that if $f_i(x,y)=0$ for all $x,y\in H$ then all inequalities in $(\ref{eq:CSEP})$ are 
automatically satisfied. Thus, the CSEP reduces to the following convex feasibility problem (CFP)
\begin{equation}\label{eq:CFP}
 \mbox{Find}~x^*\in K:=\cap_{i=1}^N K_i \ne \O
\end{equation}
which is to find an element in the intersection of a family of convex sets $\left\{K_i\right\}_{i=1}^N$ in a Hilbert space $H$. The CFP has 
received great attention due to broad applicability in many areas of applied mathematics, most notably, as image recovery from projections, 
computerized tomography, and radiation therapy treatment planing, see for instance \cite{BB1996,C1996}. Besides, if $K_i$ 
is the fixed point set of the mapping $S_i:H\to H$, then the CFP $(\ref{eq:CFP})$ is the common fixed point problem (CFPP), i.e.,
\begin{equation}\label{eq:CFPP}
 \mbox{Find}~x^*\in F:=\cap_{i=1}^N F(S_i) \ne \O,
\end{equation}
where $F(S_i)$ is the fixed point set of $S_i,~i=1,\ldots,N$. Also, if $K_i=H$ and $f_i(x,y)=\left\langle x-S_i x,y-x\right\rangle$ then it is 
easy to show that $x^*$ is a fixed point of $S_i$ if and only if it is a solution of the EP for the bifunction $f_i$ on $K_i$ \cite{BO1994}. 
Thus, the CSEP also becomes the CFPP $(\ref{eq:CFPP})$. Some parallel algorithms for solving the CFPP can be found 
in \cite{AH2014b,AH2014,H2015}.\\
If $f_i(x,y)=\left\langle A_i(x),y-x\right\rangle$, where $A_i:H \to H$ are nonlinear operators, then the 
CSEP becomes the following common solutions of variational inequalities problem (CSVIP): Find $x^*\in K:=\cap_{i=1}^N K_i$ such that
\begin{equation}\label{eq:CSVIP}
 \left\langle A_i(x^*),y-x^*\right\rangle\ge 0, ~\forall y\in K_i, ~i=1,\ldots,N
\end{equation}
 which was announced in \cite{CGRS2012}. Moreover, there are many other mathematical models which are special cases of the 
CSEP such as: common minimizer problems, common saddle point problems, variational inequalities 
over the intersection of closed convex subsets, common solutions of operator equations, 
see \cite{ABH2014,AH2014b,AH2014,BO1994,CGRS2012,H2015} and the references therein. These 
problems have been widely studied over the past decades because of their practical applications to  
image reconstruction, signal processing, biomedical engineering, communication, etc \cite{BB1996,CCCDH2011,C1996,S1987}.

In this paper, we propose two parallel and cyclic extragradient - cutting algorithms for solving the CSEP 
for pseudomonotone bifunctions. The former is extended from a previously known algorithm for variational inequalities 
\cite{CGRS2012} to equilibrium problems. The authors in \cite{CGRS2012} studied the CSVIP for 
Lipschitz continuous and monotone operators. They used the extragradient (or double projection) method 
which was introduced by Korpelevich \cite{K1976} in Euclidean space, and by Nadezhkina and Takahashi \cite{NT2006} in 
Hilbert space to construct iteration sequences. Our first algorithm 
reduces to the CSVIP under a weaker hypothesis that operators need only the pseudomonotonicity. 
The latter is a sequential algorithm which seems to be performed more easily than the first and can delvelop 
practical numerical experiments by using the known techniques of Solodov and Svaiter \cite{SS2000} 
when the number of subproblems $N$ is large. The cyclic algorithm can be considered as an improvement of 
the iterative method in \cite{CGRS2012} and others when the CSEP is reduced to the CSVIP.

The paper is organized as follows: In Section $\ref{pre}$, we collect some definitions and primary results for using in the next section. Section $\ref{main}$ deals with our proposed algorithms and proving the convergence theorems.
\section{\bf Preliminaries}\label{pre}
In this section, we recall some definitions and results for further researches. For solving the CSEP $(\ref{eq:CSEP})$, 
we assume that each bifunction $f_i$ satisfies the following conditions:
\begin{itemize}
\item[(A1).] $f_i$ is pseudomonotone on $H$, i.e., for all $x,y\in H$,
$$ f_i(x, y) \geq 0 \Rightarrow  f_i(y,x) \leq 0; $$
\item [(A2).]  $f_i$ is Lipschitz-type continuous, i.e., there exist two positive constants $c_1,c_2$ such that
$$ f_i(x,y) + f_i(y,z) \geq f_i(x,z) - c_1||x-y||^2 - c_2||y-z||^2, \quad \forall x,y,z \in H;$$
\item [(A3).]   $f_i$ is weakly continuous on $H\times H$;
\item [(A4).]  $f_i(x,.)$ is convex and subdifferentiable on $H$  for every fixed $x\in H.$
\end{itemize}  
Note that the condition $\rm (A2)$ is fulfilled for the bifunction 
$$f(x,y)=\left\langle A(x),y-x\right\rangle,$$
where $A$ is a Lipschitz continuous operator (proved in Corollary $\ref{cor.1}$ below). We have the following result.
\begin{lemma}\label{EPexistence}\cite[Proposition 4.1]{BS1996}
If the bifunction $f$ satisfies the conditions $\rm (A1)-(A4)$, then the solution set $EP(f)$ is closed and convex.
\end{lemma}
The metric projection $P_C:H\to C$ is defined by
\begin{equation*}\label{eq:1.3}
P_C x=\arg\min\left\{\left\|y-x\right\|:y\in C\right\}.
\end{equation*}
Since $C$ is nonempty, closed and convex, $P_Cx$ exists and is unique. It is also known that $P_C$ has the following characteristic properties
\begin{lemma}\label{PropertyPC}
Let $P_C:H\to C$ be the metric projection from $H$ onto $C$. Then
\begin{itemize}
\item [$(i)$] $P_C$ is firmly nonexpansive, i.e.,
\begin{equation*}\label{eq:FirmlyNonexpOfPC}
\left\langle P_C x-P_C y,x-y \right\rangle \ge \left\|P_C x-P_C y\right\|^2,~\forall x,y\in H.
\end{equation*}
\item [$(ii)$] For all $x\in C, y\in H$,
\begin{equation}\label{eq:ProperOfPC}
\left\|x-P_C y\right\|^2+\left\|P_C y-y\right\|^2\le \left\|x-y\right\|^2.
\end{equation}
\item [$(iii)$] $z=P_C x$ if and only if 
\begin{equation}\label{eq:EquivalentPC}
\left\langle x-z,z-y \right\rangle \ge 0,\quad \forall y\in C.
\end{equation}
\end{itemize}
\end{lemma}
The normal cone $N_C$ to $C$ at a point $x\in C$ is defined by
$$ N_C(x)=\left\{w\in H:\left\langle w,x-y\right\rangle \ge 0, \forall y\in C\right\}. $$
The following lemma is similarly proved to Theorem 27.4 in \cite{R1970} (also see Theorem 3.1 in \cite{DD2012}) by 
using Moreau-Rockafellar Theorem in \cite{LL1988} to find the subdifferential of a sum of convex function $g$ and indicator function $\delta_C$ to $C$ in a real Hilbert space $H$.
\begin{lemma}\cite[Theorem 27.4]{R1970}\label{lem.Equivalent_MinPro}
Let $C$ be a convex subset of a real Hilbert space H and $g:C\to \mathbb{R}$ be a convex and subdifferentiable function on $C$. Then, 
$x^*$ is a solution to the following convex problem
\begin{equation*}\min\left\{g(x):x\in C\right\}
\end{equation*}
if and only if  ~  $0\in \partial g(x^*)+N_C(x^*)$, where $\partial g(.)$ denotes the subdifferential of $g$ and $N_C(x^*)$ is the normal cone 
of ~ $C$ at $x^*$.
\end{lemma}
\section{\bf Main results}\label{main}
In this section, we propose two algorithms for solving the CSEP $(\ref{eq:CSEP})$ and analyse the convergence of 
the iteration sequences generated by the algorithms. In the sequel, without loss of generality, we assume that the 
bifunctions $f_i,i=1,\ldots,N$ are Lipschitz-type continuous with the same positive constants $c_1$ and $c_2$, i.e.,
$$ f_i(x,y) + f_i(y,z) \geq f_i(x,z) - c_1||x-y||^2 - c_2||y-z||^2 $$
for all $x,y,z \in H$. Moreover, the solution set $F=\cap_{i=1}^N EP(f_i)$ is nonempty.
\begin{algorithm}\label{algor1}(The parallel hybrid extragradient-cutting algorithm)\\
\textbf{Initialize.} $x_0\in H, n:=0$, $0<\lambda\le\lambda_k^i\le \mu<\min\left\{\frac{1}{2c_1},\frac{1}{2c_2}\right\}, \gamma_k^i\in [\epsilon,\frac{1}{2}]$ for some $\epsilon \in (0,\frac{1}{2}]$, $k=1,2,\ldots$ and $i=1,\ldots,N$.\\
\textbf{Step 1.} Solve $N$ strongly convex problems in parallel, $i=1,\ldots,N$
$$ y_n^i=\arg\min\left\{\lambda_n^i f_i(x_n,y)+\frac{1}{2}||x_n-y||^2: y\in K_i\right\}. $$
\textbf{Step 2.} Solve $N$ strongly convex problems in parallel, $i=1,\ldots,N$
$$ z_n^i=\arg\min\left\{\lambda_n^i f_i(y_n^i,y)+\frac{1}{2}||x_n-y||^2: y\in K_i\right\}. $$
\textbf{Step 3.} Determine the next approximation $x_{n+1}$ as the projection of $x_0$ onto the intersection $H_n\cap W_n$
$$ x_{n+1}=P_{H_n\cap W_n}(x_0), $$
where $H_n=\cap_{i=1}^N H_n^i$ and 
\begin{eqnarray*}
&&H_n^i=\left\{z\in H:\left\langle x_n-z_n^i, z-x_n -\gamma_n^i(z_n^i-x_n)\right\rangle\le 0\right\},\\
&&W_n=\left\{z\in H: \left\langle x_0-x_n,x_n-z\right\rangle\ge 0\right\}.
\end{eqnarray*}
\textbf{Step 4.} If $x_{n+1}=x_n$ then stop. Otherwise, set $n:=n+1$ and go back \textbf{Step 1.}
\end{algorithm}
In order to prove the convergence of Algorithm $\ref{algor1}$, we need the following lemmas.
\begin{lemma}\label{lem.1}\cite[Lemma 3.1]{A2013} {\rm (cf. \cite[Theorem 3.2]{QMH2008})}
Assume that $x^*\in F$. Let $\left\{y_n^i\right\},\left\{z_n^i\right\}$ be the sequences determined as in Steps 1 and 2 of Algorithm $\ref{algor1}$. Then, there holds the relation
\begin{equation*}\label{eq:}
||z_n^i-x^*||^2\le||x_n-x^*||^2 - \left(1-2\lambda_n^i c_1\right)||y_n^i-x_n||^2-\left(1-2\lambda_n^i c_2\right)||z_n^i-y_n^i||^2.
\end{equation*}
\end{lemma}
\begin{lemma}\label{lem.2}
If Algorithm $\ref{algor1}$ reaches to the iteration step $n$, then $F\subset H_n\cap W_n$ and $x_{n+1}$ is well-defined.
\end{lemma}
\begin{proof}
By Lemma $\ref{EPexistence}$, the solution set $F$ is closed and convex. From the definitions of $H_n^i,W_n,i=1,\ldots,N$, we see that these sets are closed and convex. Thus, $H_n$ is also closed and convex. We now show that $F\subset H_n\cap W_n$ for all $n\ge 0$. For each $i=1,\ldots,N$, putting
$$ C_n^i=\left\{z\in H:||z-z_n^i||\le ||z-x_n||\right\}. $$
A straightforward calculation leads to
$$ C_n^i=\left\{z\in H:\left\langle x_n-z_n^i, z-x_n -\frac{1}{2}(z_n^i-x_n)\right\rangle\le 0\right\}.$$
By $\gamma_n^i\in[\epsilon,\frac{1}{2}]$, $C_n^i\subset H_n^i$ for all $i=1,\ldots,N$. So, $C_n:=\cap_{i=1}^N C_n^i \subset H_n$. From Lemma $\ref{lem.1}$ and $0<\lambda\le\lambda_n^i\le\mu<\min\left\{\frac{1}{2c_1},\frac{1}{2c_2}\right\}$, we obtain $||z_n^i-x^*||\le||x_n-x^*||$ for all $x^*\in F$ and $i=1,\ldots,N$. This implies that $F\subset C_n^i$. Therefore, $F\subset C_n$ for all $n\ge 0$. Next, we show that $F\subset C_n \cap W_n$ for all $n\ge 0$ by the induction. Indeed, we have $F\subset C_0\cap W_0$. Assume that $F\subset C_n \cap W_n$ for some $n\ge 0$. From $x_{n+1}=P_{H_n\cap W_n}(x_0)$ and $(\ref{eq:EquivalentPC})$, we obtain
$$ \left\langle x_0-x_{n+1},x_{n+1}-z\right\rangle\ge 0, ~\forall z\in H_n\cap W_n.$$
Since $F\subset C_n \cap W_n \subset H_n \cap W_n$,
$$ \left\langle x_0-x_{n+1},x_{n+1}-z\right\rangle\ge 0, ~\forall z\in F.$$
This together with the definition of $W_{n+1}$ implies that $F\subset W_{n+1}$, and so $F\subset C_{n+1}\cap W_{n+1}$. Thus, 
by the induction we obtain $F \subset C_n\cap W_n$ for all $n\ge 0$. By $C_n\subset H_n$, we get $F\subset H_n\cap W_n$ for all 
$n\ge 0$. Since $F$ is nonempty, $H_n\cap W_n$ is also nonempty. Therefore, $x_{n+1}$ is well-defined.
\end{proof}
\begin{lemma}\label{lem.3}
If Algorithm $\ref{algor1}$ finishes at the iteration step $n<\infty$, then $x_n\in F$.
\end{lemma}
\begin{proof}
Assume that $x_{n+1}=x_n$. Since $x_{n+1}=P_{H_n\cap W_n}(x_0)$, $x_n=x_{n+1}\in H_n$. This together with the definition of $H_n$ implies that $\gamma_n^i||x_n-z_n^i||\le 0$. From the last inequality and $\gamma_n^i\ge \epsilon>0$, one gets $x_n=z_n^i$. By Lemma $\ref{lem.1}$ and the hypothesis of $\lambda_n^i$, we obtain $y_n^i=x_n$. Thus
$$ x_n=\arg\min\left\{\lambda_n^i f_i(x_n,y)+\frac{1}{2}||x_n-y||^2: y\in K_i\right\}. $$
Thus, from \cite[Proposition 2.1]{M2000}, one has $x_n\in EP(f_i)$ for all $i=1,\ldots, N,$ or $x_n \in F$. The proof of Lemma  $\ref{lem.3}$ is complete.
\end{proof}
\begin{lemma}\label{lem.4}
Let $\left\{x_n\right\},\left\{y_n^i\right\},\left\{z_n^i\right\}$ be the sequences generated by Algorithm $\ref{algor1}$. Then, there hold the following relations for all $i=1,\ldots,N$
\begin{equation*}\label{eq:}
\lim_{n\to\infty}||x_{n+1}-x_n||=\lim_{n\to\infty}||y_n^i-x_n||=\lim_{n\to\infty}||z_n^i-x_n||=0.
\end{equation*}
\end{lemma}
\begin{proof}
From the definition of $W_n$ and the relation $(\ref{eq:EquivalentPC})$, we have $x_n=P_{W_n}(x_0)$. For each $u\in F\subset W_n$, 
from $(\ref{eq:ProperOfPC})$, one obtains
\begin{equation}\label{eq:1}
||x_n-x_0||\le ||u-x_0||.
\end{equation}
Thus, the sequence $\left\{||x_n-x_0||\right\}$ is bounded, and so, from Lemma $\ref{lem.1}$ the sequences $\left\{x_n\right\}$ and $\left\{z_n^i\right\}$ are also bounded. Moreover, the projection $x_{n+1}=P_{H_n\cap W_n}(x_0)$ implies $x_{n+1}\in W_n$. Thus, from $x_n=P_{W_n}x_0$ and $(\ref{eq:ProperOfPC})$, we also see that
$$ ||x_n-x_0||\le||x_{n+1}-x_0||. $$
So, the sequence $\left\{||x_n-x_0||\right\}$ is non-decreasing. Hence, there exists the limit of the sequence $\left\{||x_n-x_0||\right\}$. By $x_{n+1}\in W_n$, $x_n=P_{W_n}(x_0)$ and the relation $(\ref{eq:ProperOfPC})$, we also have
\begin{equation}\label{eq:2}
||x_{n+1}-x_n||^2\le ||x_{n+1}-x_0||^2-||x_n-x_0||^2
\end{equation}
Passing to the limit in the inequality $(\ref{eq:2})$ as $n\to\infty$, one gets
\begin{equation}\label{eq:3}
\lim_{n\to\infty}||x_{n+1}-x_n||=0.
\end{equation}
Since $x_{n+1}\in H_n$, $x_{n+1}\in H_n^i$ for all $i=1,\ldots,N$. From the definition of $H_n^i$, we have
$$\gamma_n^i||z_n^i-x_n||^2\le\left\langle x_n-z_n^i,x_n-x_{n+1}\right\rangle.$$
This together with the inequality $|\left\langle x,y\right\rangle|\le ||x||||y||$ implies that $\gamma_n^i||z_n^i-x_n||\le ||x_n-x_{n+1}||$. From $\gamma_n^i\ge \epsilon>0$ and $(\ref{eq:3})$, one has
\begin{equation}\label{eq:4}
\lim_{n\to\infty}||z_{n}^i-x_n||=0, ~i=1,\ldots,N.
\end{equation}
From Lemma $\ref{lem.1}$ and the triangle inequality, we have
\begin{eqnarray*}
\left(1-2\lambda_n^i c_1\right)||y_n^i-x_n||^2&\le&||x_n-x^*||^2-||z_n^i-x^*||^2\\ 
&\le& (||x_n-x^*||+||z_n^i-x^*||)(||x_n-x^*||-||z_n^i-x^*||)\\
&\le& (||x_n-x^*||+||z_n^i-x^*||)||x_n-z_n^i||.
\end{eqnarray*}
The last inequality together with $(\ref{eq:4})$, the hypothesis of $\lambda_n^i$ and the boundedness of $\left\{x_n\right\}$, $\left\{z_n^i\right\}$ implies that
$$
\lim_{n\to\infty}||y_{n}^i-x_n||=0, ~i=1,\ldots,N.
$$
The proof Lemma $\ref{lem.3}$ is complete.
\end{proof}
\begin{theorem}\label{theo.1}
Assume that the bifunctions $f_i,i=1,\ldots,N$ satisfy all conditions $\rm (A1)-(A4)$. In addition the solution set $F$ is nonempty. Then, 
the sequences $\left\{x_n\right\},\left\{y_n^i\right\},\left\{z_n^i\right\}$ generated by Algorithm $\ref{algor1}$ converge strongly to $P_F(x_0)$.
\end{theorem}
\begin{proof}
By Lemmas $\ref{EPexistence}$ and $\ref{lem.2}$, we see that the sets $F,H_n,W_n$ are closed and convex for all $n\ge 0$. Besides, by Lemma $\ref{lem.4}$ the sequence $\left\{x_n\right\}$ is bounded. Assume that $p$ is any weak cluster point of the sequence $\left\{x_n\right\}$. Then, there exists a subsequence of $\left\{x_n\right\}$ converging weakly to $p$. For the sake of simplicity, we denote this subsequence again by $\left\{x_n\right\}$ and $x_n\rightharpoonup p$ as $n\to \infty$. We now show that $p\in F$. Indeed, from the relation
\begin{equation}\label{eq:5}
y_n^i = {\rm argmin} \{\lambda_n^i f_i(x_n, y) +\frac{1}{2}||x_n-y||^2:  y \in K_i\},
\end{equation}
and Lemma $\ref{lem.Equivalent_MinPro}$, one gets
\begin{equation}\label{eq:6}
0\in \partial_2 \left\{\lambda_n^i f_i(x_n,y)+\frac{1}{2}||x_n-y||^2\right\}(y_n^i)+N_{K_i}(y_n^i).
\end{equation}
Thus, there exist $\bar{w}\in N_{K_i}(y_n^i)$ and $w\in \partial_2 f_i (x_n,y_n^i)$  such that
\begin{equation}\label{eq:10}
\lambda_n^i w+x_n-y_n^i+\bar{w}=0.
\end{equation}
From the definition of the normal cone $N_{K_i}(y_n^i)$, we have $\left\langle \bar{w}, y-y_n^i\right\rangle \le 0$ for all $y\in K_i$. Taking into account $(\ref{eq:10})$, we obtain
\begin{equation}\label{eq:11}
\lambda_n^i \left\langle w, y-y_n^i\right\rangle \ge \left\langle y_n^i-x_n, y-y_n^i\right\rangle
\end{equation}
for all $y\in K_i$. Since $w\in \partial_2 f_i (x_n,y_n^i)$, 
\begin{equation}\label{eq:12}
f_i(x_n,y)-f_i(x_n,y_n^i)\ge \left\langle w, y-y_n^i\right\rangle, \forall y\in K_i.
\end{equation}
Combining $(\ref{eq:11})$ and $(\ref{eq:12})$, one has
\begin{equation}\label{eq:13}
\lambda_n^i \left(f_i(x_n,y)-f_i(x_n,y_n^i)\right) \ge \left\langle y_n^i-x_n, y-y_n^i\right\rangle, \forall y\in K_i.
\end{equation}
From $||y_n^i-x_n||\to 0$ and $x_n\rightharpoonup p$, we also have $y_n^i\rightharpoonup p$. Passing to the limit in the inequality $(\ref{eq:13})$ as $n\to\infty$ and employing the assumption (A3) and $\lambda_n^i\ge\lambda>0$, we conclude that $f_i(p,y)\ge 0$ for all $y \in K_i, i=1,\ldots,N$. Hence, $p\in F$. Finally, we show that $x_n\to p$. Putting $x^\dagger=P_F(x_0)$. Using the inequality $(\ref{eq:1})$ with $u=x^\dagger$, we get
$$ ||x_n-x_0||\le ||x^\dagger-x_0||.$$
By the weak lower semicontinuity of the norm $||.||$ and $x_n\rightharpoonup p$, we have
\begin{equation*}
||p-x_0||\le \lim\inf_{n\to\infty}||x_{n}-x_0||\le \lim\sup_{n\to\infty}||x_{n}-x_0||\le||x^\dagger-x_0||.
\end{equation*}
By the definition of $x^\dagger$, $p=x^\dagger$ and so $\lim_{n\to\infty}||x_{n}-x_0||=||x^\dagger-x_0||$. Thus, $\lim_{n\to\infty}||x_{n}||=||x^\dagger||$. By the Kadec-Klee property of the Hilbert space $H$, we have $x_{n}\to x^\dagger=P_Fx_0$ as $n\to\infty$. From Lemma $\ref{lem.4}$, one also obtains that $\left\{y_n^i\right\},\left\{z_n^i\right\}$ converge strongly $P_Fx_0$. This completes the proof of Theorem $\ref{theo.1}$.
\end{proof}
Using Theorem $\ref{theo.1}$, we get the following result which obtained in \cite{CGRS2012}.
\begin{corollary}\label{cor.1}
Let $A_i,i=1,\ldots,N$ be $L$ - Lipschitz continuous and pseudomonotone mappings from a real Hilbert space $H$ to itself. 
In addition, the solution set $\bar{F}=\cap_{i=1}^N VI(A_i,K_i)$ is nonempty, where $VI(A_i,K_i)$ stands for the solution 
set of the variational inequality which is to find $x^*\in K_i$ such that $\left\langle A_i(x^*),y-x^*\right\rangle\ge 0,~\forall y\in K_i$. 
Let $\left\{x_n\right\},\left\{y_n^i\right\},\left\{z^i_n\right\}$ be the sequences generated by the following parallel manner
\begin{equation*}\label{eq:}
\left \{
\begin{array}{ll}
x_0\in H,\\
y_n^i=P_{K_i}(x_n-\lambda_n^iA_i(x_n)),\\ 
z_n^i=P_{K_i}(x_n-\lambda_n^iA_i(y^i_n)),\\
H_n^i=\left\{z\in H:\left\langle x_n-z_n^i, z-x_n -\gamma_n^i(z_n^i-x_n)\right\rangle\le 0\right\},\\
H_n=\cap_{i=1}^N H_n^i,\\
W_n=\left\{z\in H: \left\langle x_0-x_n,x_n-z\right\rangle\ge 0\right\},\\
x_{n+1}=P_{H_n \cap W_n}x_0,n\ge 0,
\end{array}
\right.
\end{equation*}
where $0<\lambda\le\lambda_n^i\le\mu<1/L$, $0<\epsilon\le \gamma_n^i\le 1/2$ for some $\epsilon \in (0,1/2]$. Then, the sequences $\left\{x_n\right\},\left\{y_n^i\right\},\left\{z^i_n\right\}$ converge strongly to $P_{\bar{F}} x_0$.
\end{corollary}
\begin{proof}
For each $i=1,\ldots,N$, putting $f_i(x,y)=\left\langle A_i(x),y-x\right\rangle$. Since $A_i$ is pseudomonotone, $f_i$ is too. So the condition $\rm (A1)$ is satisfied for each $f_i$. The conditions $\rm (A3), (A4)$ are automatically fulfilled. We now show that $f_i$ satisfies the condition $\rm (A2)$. Indeed, from the $L$ - Lipschitz continuty of $A_i$, we have
\begin{eqnarray*}
f_i(x,y)+f_i(y,z)-f_i(x,z)&=&\left\langle A_i(x),y-x\right\rangle+\left\langle A_i(y),z-y\right\rangle\\
&&-\left\langle A_i(x),z-x\right\rangle\\ 
&=&\left\langle A_i(x),y-z\right\rangle+\left\langle A_i(y),z-y\right\rangle\\
& =&\left\langle A_i(x)-A_i(y),y-z\right\rangle\\
&\ge&-||A_i(x)-A_i(y)||||y-z||\\
&\ge&-L||x-y||||y-z||\\
&\ge&-\frac{L}{2}||x-y||^2-\frac{L}{2}||y-z||^2.
\end{eqnarray*}
This implies that $f_i$ satisfies the condition $\rm (A2)$ with $c_1=c_2=L/2$. 
From Algorithm $\ref{algor1}$ , we have
\begin{align*}
&y_n^i = {\rm argmin} \{ \lambda_n^i\left\langle A_i(x_n),y-x_n\right\rangle +\frac{1}{2}||x_n-y||^2:  y \in K_i\},\\ 
&z_n^i = {\rm argmin} \{ \lambda_n^i\left\langle A_i(y_n^i),y-y_n^i\right\rangle +\frac{1}{2}||x_n-y||^2:  y \in K_i\}.
\end{align*}
A straightforward computation yields
\begin{align*}
&y_n^i = {\rm argmin} \{\frac{1}{2}||y-(x_n-\lambda_n^i A_i(x_n) )||^2:  y \in K_i\}=P_{K_i}(x_n-\lambda_n^i A_i(x_n) ),\\ 
&z_n^i = {\rm argmin} \{ \frac{1}{2}||y-(x_n-\lambda_n^i A_i(y^i_n) )||^2:  y \in K_i\}=P_{K_i}(x_n-\lambda_n^i A_i(y^i_n)).
\end{align*}
Apply Theorem $\ref{theo.1}$ to Corollary $\ref{cor.1}$, we come to the desired result.
\end{proof}
\begin{remark}
In Corollary $\ref{cor.1}$, we need only the pseudomonotonicity of the mappings $A_i,i=1,\ldots,N$ to obtain the convergence of the iteration sequences. However, in order to get the same result, Censor et al \cite{CGRS2012} required the monotonicity of these mappings which is more strict than the pseudomonotonicity.
\end{remark}
In Algorithm $\ref{algor1}$, at the $n^{th}$ step, in order to determine the next approximation $x_{n+1}$ we have to construct $N+1$ the subsets $H_n^i,i=1,\ldots,N$ and $W_n$ and solve the following optimization problem on the intersection of $N+1$ closed convex sets
\begin{equation*}\label{eq:Optimization}
\left \{
\begin{array}{ll}
\min||z-x_0||^2,\\
\mbox{such that}\quad z\in H_n^1\cap \ldots \cap H_n^N \cap W_n. 
\end{array}
\right.
\end{equation*}
This seems very costly when the number of subproblems $N$ is large. Thus, Algorithm $\ref{algor1}$ can not develop 
practical numerical experiments. To overcome the complexity of this algorithm. We next propose the following cyclic 
algorithm for solving the CSEP for pseudomonotone bifunctions $f_i,i=1,\ldots,N$. We denote $[n]=n (mod~N)+1$ to 
stand for the mod function taking values in $\left\{1,2,\ldots,N\right\}$.
\begin{algorithm}\label{algor2}(The cyclic hybrid extragradient-cutting algorithm)\\
\textbf{Initialize.} $x_0\in H$, n:=0, $0<\lambda\le\lambda_k\le \mu<\min\left\{\frac{1}{2c_1},\frac{1}{2c_2}\right\}, \gamma_k\in [\epsilon,\frac{1}{2}]$ for some $\epsilon \in (0,\frac{1}{2}]$ and $k=1,2,\ldots$.\\
\textbf{Step 1.} Solve the strongly convex problem
$$ y_n=\arg\min\left\{\lambda_nf_{[n]}(x_n,y)+\frac{1}{2}||x_n-y||^2: y\in K_{[n]}\right\}. $$
\textbf{Step 2.} Solve the strongly convex problem 
$$ z_n=\arg\min\left\{\lambda_n f_{[n]}(y_n,y)+\frac{1}{2}||x_n-y||^2: y\in K_{[n]}\right\}. $$
\textbf{Step 3.} Determine the next approximation $x_{n+1}$ as the projection of $x_0$ onto $H_n\cap W_n$
$$ x_{n+1}=P_{H_n\cap W_n}(x_0), $$
where
\begin{eqnarray*}
&&H_n=\left\{z\in H:\left\langle x_n-z_n, z-x_n -\gamma_n(z_n-x_n)\right\rangle\le 0\right\},\\
&&W_n=\left\{z\in H: \left\langle x_0-x_n,x_n-z\right\rangle\ge 0\right\}.
\end{eqnarray*}
\textbf{Step 4.} Set $n:=n+1$ and go back \textbf{Step 1.}
\end{algorithm}
Using the same technique as in \cite[Algorithm 1]{SS2000}, we can find the explicit formula of the projection $x_{n+1}$ of $x_0$ 
onto the intersection of two subsets $H_n$ and $W_n$ in Step 3 of Algorithm $\ref{algor2}$. Indeed, from the definitions 
of $H_n$ and $W_n$, we see that they are either halfspaces or $H$. Putting $v_n=x_n+\gamma_n(z_n-x_n)$, we rewrite the 
set $H_n$ as follows
$$ H_n=\left\{z\in H:\left\langle x_n-z_n, z-v_n\right\rangle\le 0\right\}. $$
By analysing similarly as in \cite[Algorithm 1]{SS2000}, we get the explicit formula of the projection $x_{n+1}$ of $x_0$ onto $H_n\cap W_n$
$$
x_{n+1}:=P_{H_n}x_0 =
\left\{
\begin{array}{ll}
x_0&\mbox{if}\qquad z_n=x_n,\\ 
x_0-\frac{\left\langle x_n-z_n,x_0-v_n\right\rangle}{||x_n-z_n||^2}(x_n-z_n)&\mbox{if}\qquad z_n\ne x_n.\\
\end{array}
\right.
$$
if $P_{H_n}x_0 \in W_n$. Otherwise,
$$
x_{n+1}=x_0+t_1(x_n-z_n)+t_2(x_0-x_n),
$$
where $t_1,t_2$ is the solution of the system of linear equations with two unknowns
$$
\left\{
\begin{array}{ll}
t_1 ||x_n-z_n||^2+t_2\left\langle x_n-z_n,x_0-x_n \right\rangle=-\left\langle x_0-v_n, x_n-z_n\right\rangle,\\
t_1\left\langle x_n-z_n,x_0-x_n \right\rangle+t_2 ||x_0-x_n||^2=-||x_0-x_n||^2.
\end{array}
\right.
$$
\begin{theorem}\label{theo.2}
Assume that the bifunctions $f_i,i=1,\ldots,N$ satisfy all conditions $\rm (A1)-(A4)$. In addition, the solution set $F$ is nonempty. 
Then, the sequences $\left\{x_n\right\},\left\{y_n\right\},\left\{z_n\right\}$ generated by Algorithm $\ref{algor2}$ converge strongly to $P_F(x_0)$.
\end{theorem}
\begin{proof}
By the same arguments as in the proof of Lemmas $\ref{lem.1}-\ref{lem.4}$, we see that $F, H_n, W_n$ are closed and convex, and $F\subset H_n\cap W_n$ for all $n\ge 0$. Besides, the sequence $\left\{x_n\right\}$ is bounded and there hold the relations
\begin{equation*}\label{eq:14}
\lim_{n\to\infty}||x_{n+1}-x_n||=\lim_{n\to\infty}||y_n-x_n||=\lim_{n\to\infty}||z_n-x_n||=0.
\end{equation*}
Assume that $p$ is any weak cluster point of the sequence $\left\{x_n\right\}$. For each fixed index $i\in \left\{1,2,\ldots,N\right\}$, since the 
set of indexes $i$ is finite, by \cite[Theorem 5.3]{BC2001} there exists a subsequence $\left\{x_{n_j}\right\}$ of $\left\{x_n\right\}$ such that $x_{n_j}\rightharpoonup p$ as $j\to \infty$ and $[n_j]=i$ for all $j$. Repeat the proofs of $(\ref{eq:5})-(\ref{eq:13})$, we also conclude that $p\in EP(f_i)$. This is true for all $i=1,\ldots,N$. Thus, $p\in F$. The rest of the proof of Theorem $\ref{theo.2}$ is same to that one of Theorem \ref{theo.1}. 
\end{proof}
\begin{corollary}\label{cor.2}
Let $A_i,i=1,\ldots,N$ be $L$ - Lipschitz continuous and pseudomonotone mappings from a real Hilbert space $H$ to itself. 
In addition, the solution set $\bar{F}=\cap_{i=1}^N VI(A_i,K_i)$ is nonempty, where $VI(A_i,K_i)$ is defined as in 
Corollary $\ref{cor.1}$. Let $\left\{x_n\right\},\left\{y_n\right\},\left\{z_n\right\}$ be the sequences generated by 
the following cyclic manner
\begin{equation*}\label{eq:}
\left \{
\begin{array}{ll}
x_0\in H,\\
y_n=P_{K_{[n]}}(x_n-\lambda_nA_{[n]}(x_n)),\\ 
z_n=P_{K_{[n]}}(x_n-\lambda_nA_{[n]}(y_n)),\\
H_n=\left\{z\in H:\left\langle x_n-z_n, z-x_n -\gamma_n(z_n-x_n)\right\rangle\le 0\right\},\\
W_n=\left\{z\in H: \left\langle x_0-x_n,x_n-z\right\rangle\ge 0\right\},\\
x_{n+1}=P_{H_n \cap W_n}x_0,
\end{array}
\right.
\end{equation*}
where $0<\lambda\le\lambda_n\le\mu<1/L$, $0<\epsilon\le \gamma_n\le 1/2$ for some $\epsilon \in (0,1/2]$. Then, the sequences $\left\{x_n\right\},\left\{y_n\right\},\left\{z_n\right\}$ converge strongly to $P_{\bar{F}} x_0$.
\end{corollary}
\begin{proof}
Using Theorem $\ref{theo.2}$ and arguing similarly as in the proof of Corollary $\ref{cor.1}$, we lead to the desired conclusion.
\end{proof}
\begin{remark}
Corollaries $\ref{cor.1}$ and $\ref{cor.2}$  with $N=1$ give us the corresponding result of Nadezhkina and Takahashi in \cite[Theorem 4.1]{NT2006}.
\end{remark}
\section*{\bf Acknowledgments} The author wishes to thank the referees for their valuable comments and suggestions which improved the paper. The author also thanks Prof.Dsc Pham Ky Anh, Department of Mathematics, Hanoi University of Science, VNU for his useful hints to complete this work.


\begin{thebibliography}{20}
\bibitem{A2011} P.N. Anh, A hybrid extragradient method extended to fixed point problems and equilibrium problems, \textit{
Optimization} \textbf{62} (2013), no. 2, 271--283.
\bibitem{A2013}\bysame, A hybrid extragradient method for pseudomonotone equilibrium
problems and fixed point problems, \textit{Bull. Malays. Math. Sci. Soc.} \textbf{36} (2013), no. 1, 107--116.
\bibitem{ABH2014}P. K. Anh, Ng. Buong and D. V. Hieu, Parallel methods for regularizing systems of equations involving accretive operators, \textit{Appl. Anal.} \textbf{93} (2014), no. 10, 2136--2157.
\bibitem{AH2014b} P. K. Anh, D. V. Hieu, Parallel hybrid methods for variational inequalities, equilibrium problems and common fixed point problems, \textit{Vietnam J. Math.} (2015), DOI:10.1007/s10013-015-0129-z.
\bibitem{AH2014}\bysame, Parallel and sequential hybrid methods for a finite family of asymptotically quasi $\phi$
-nonexpansive mappings, \textit{J. Appl. Math. Comput.} \textbf{48} (2015), 241--263.
\bibitem{BB1996} H.H. Bauschke and J.M. Borwein, On projection algorithms for solving convex feasibility problems, \textit{SIAM Review} \textbf{38} (1996), 367-- 426.
\bibitem{BC2001} H.H. Bauschke and P.L. Combettes, A weak-to-strong convergence principle for Fejer monotone methods in Hilbert spaces, \textit{Math. Oper. Res.} \textbf{26} (2001), 248--264.
\bibitem{BS1996} M. Bianchi, S. Schaible, Generalized monotone bifunctions and equilibrium problems, \textit{J. Optim. Theory Appl.} 
\textbf{90} (1996), 31--43.
\bibitem{BO1994} E. Blum, W. Oettli, From optimization and variational inequalities to equilibrium problems, \textit{Math. Student} \textbf{63} (1994), 123--146.
\bibitem{CCCDH2011}Y. Censor, W. Chen, P. L. Combettes, R. Davidi, G. T. Herman, On the effectiveness of projection methods for convex feasibility problems 
with linear inequality constraints, \textit{Comput. Optim. Appl.} (2011), DOI: 10.1007/s10589-011-9401-7.
\bibitem{CGRS2012}Y. Censor, A. Gibali, S. Reich, S. Sabach, Common solutions to variational inequalities, \textit{Set-Valued Var. Anal.} \textbf{20} (2012), 229--247.
\bibitem{C1996}P. L. Combettes, The convex feasibility problem in image recovery. in, \textit{P.Hawkes(Ed.), Advances in Imaging and Electron Physics, Academic Press, New York}, \textbf{95} (1996), 155--270.
\bibitem{DD2012}A. Dhara, J. Dutta, \textit{Optimality conditions in convex optimization: A finite - dimensional view}, CRC Press, Taylor \& Francis, 2012 
\bibitem{DHM2014} B.V. Dinh, P.G. Hung, L.D. Muu, Bilevel optimization as a regularization approach to pseudomonotone equilibrium problems, \textit{Numer. Funct. Anal.  Optim.} \textbf{35} (2014), no. 5, 539--563.
\bibitem{H2015} D. V. Hieu, A parallel hybrid method for equilibrium problems, variational inequalities and nonexpansive mappings in Hilbert space, \textit{J. Korean Math. Soc.} \textbf{52} (2015), 373--388.
\bibitem{K1976}G. M. Korpelevich, The extragradient method for finding saddle points and other problems, \textit{Ekonomikai Matematicheskie Metody} \textbf{12} (1976), 747-756.
\bibitem{LL1988}H. C. Lai, L. J. Lin, Moreau-Rockafellar type theorem for convex set functions, \textit{J. Math. Anal. Appl.} \textbf{132} (1988), no 2, 558--571.
\bibitem{M2000}G. Mastroeni, On auxiliary principle for equilibrium problems, \textit{Publ. Dipart. Math. Univ. Pisa} \textbf{3} (2000), 1244--1258.
\bibitem{MO1992} L.D. Muu and W. Oettli, Convergence of an adative penalty scheme for finding constrained equilibria, \textit{Nonlinear Anal. TMA} \textbf{18} (1992), no. 12, 1159--1166.
\bibitem{NT2006} N. Nadezhkina and W. Takahashi, Strong convergence theorem by a hybrid method for nonexpansive mappings and Lipschitz-continuous monotone
mappings, \textit{SIAM J. Optim.} \textbf{16} (2006), 1230--1241.
\bibitem{QMH2008}T.D. Quoc, L.D. Muu and N.V. Hien, Extragradient algorithms extended to equilibrium problems, \textit{Optimization} \textbf{57} (2008), 749--776.
\bibitem{R1970}R.T. Rockafellar, \textit{Convex analysis}, Princeton, NJ: Princeton University Press, 1970.
\bibitem{SS2000}M. V. Solodov and B. F. Svaiter, Forcing strong convergence of proximal point iterations in Hilbert space, \textit{Math. Program.} \textbf{87} (2000), 189--202.
\bibitem{S1987} H .Stark (Ed.), \textit{Image recovery theory and applications, Academic Press}, Orlando, 1987.
\bibitem{TT2007}S. Takahashi, W. Takahashi, Viscosity approximation methods for equilibrium problems and fixed point in Hilbert space, \textit{J. Math. Anal. Appl.} \textbf{331} (2007), no. 1, 506--515.
\end{thebibliography}
\end{document}